\documentclass[a4paper,centertags,fleqn,reqno,12pt]{amsart}
\usepackage{amsfonts}
\usepackage[left=1.5cm, bottom=3cm]{geometry}
\usepackage{amssymb,amsmath,latexsym}

\newtheorem{thm}{Theorem}
\newtheorem{lem}{Lemma}

\theoremstyle{definition}
\newtheorem{dfn}{Definition}
\theoremstyle{remark}
\numberwithin{equation}{section}
\allowdisplaybreaks

\begin{document}
\title[New subclass of $p$-valently close-to-convex functions]{Certain
properties of a new subclass of analytic and $p$-valently close-to-convex
functions}
\author{Serap BULUT}
\address{Kocaeli University, Faculty of Aviation and Space Sciences,
Arslanbey Campus, 41285 Kocaeli, Turkey}
\email{serap.bulut@kocaeli.edu.tr}

\begin{abstract}
In the present paper we introduce and investigate an interesting subclass $%
\mathcal{K}_{s}^{\left( k\right) }\left( \gamma ,p\right) $ of analytic and $%
p$-valently close-to-convex functions in the open unit disk $\mathbb{U}$.
For functions belonging to this class, we derive several properties as the
inclusion relationships and distortion theorems. The various results
presented here would generalize many known recent results.
\end{abstract}

\maketitle

\section{Introduction}

Let $\mathcal{A}_{p}$ denote the class of all functions of the form%
\begin{equation}
f\left( z\right) =z^{p}+\sum_{n=1}^{\infty }a_{n+p}z^{n+p}\qquad \left( p\in
\mathbb{N}:=\left\{ 1,2,\ldots \right\} \right)  \label{1.1}
\end{equation}%
which are analytic in the open unit disk%
\begin{equation*}
\mathbb{U}=\left\{ z:z\in \mathbb{C}\quad \text{and}\quad \left\vert
z\right\vert <1\right\} .
\end{equation*}%
In particular, we write $\mathcal{A}_{1}=\mathcal{A}.$

For two functions $f$ and $\Theta $, analytic in $\mathbb{U},$ we say that
the function $f$ is subordinate to $\Theta $ in $\mathbb{U},$ and write%
\begin{equation*}
f\left( z\right) \prec \Theta \left( z\right) \qquad \left( z\in \mathbb{U}%
\right) ,
\end{equation*}%
if there exists a Schwarz function $\omega ,$ which is analytic in $\mathbb{U%
}$ with%
\begin{equation*}
\omega \left( 0\right) =0\qquad \text{and\qquad }\left\vert \omega \left(
z\right) \right\vert <1\quad \left( z\in \mathbb{U}\right)
\end{equation*}%
such that%
\begin{equation*}
f\left( z\right) =\Theta \left( \omega \left( z\right) \right) \quad \left(
z\in \mathbb{U}\right) .
\end{equation*}%
Indeed, it is known that%
\begin{equation*}
f\left( z\right) \prec \Theta \left( z\right) \quad \left( z\in \mathbb{U}%
\right) \Rightarrow f\left( 0\right) =\Theta \left( 0\right) \quad \text{%
and\quad }f\left( \mathbb{U}\right) \subset \Theta \left( \mathbb{U}\right) .
\end{equation*}%
Furthermore, if the function $\Theta $ is univalent in $\mathbb{U},$ then we
have the following equivalence%
\begin{equation*}
f\left( z\right) \prec \Theta \left( z\right) \quad \left( z\in \mathbb{U}%
\right) \Leftrightarrow f\left( 0\right) =\Theta \left( 0\right) \quad \text{%
and\quad }f\left( \mathbb{U}\right) \subset \Theta \left( \mathbb{U}\right) .
\end{equation*}

A function $f\in \mathcal{A}_{p}$ is said to be $p$-valently starlike of
order $\gamma \,\left( 0\leq \gamma <p\right) $ in $\mathbb{U}$ if it
satisfies the inequality%
\begin{equation*}
\Re \left( \frac{zf^{\prime }\left( z\right) }{f\left( z\right) }\right)
>\gamma \qquad \left( z\in \mathbb{U}\right)
\end{equation*}%
or equivalently%
\begin{equation*}
\frac{zf^{\prime }\left( z\right) }{f\left( z\right) }\prec \frac{p+\left(
p-2\gamma \right) z}{1-z}\qquad \left( z\in \mathbb{U}\right) .
\end{equation*}%
The class of all $p$-valent starlike functions of order $\gamma $ in $%
\mathbb{U}$ is denoted by $\mathcal{S}_{p}^{\ast }\left( \gamma \right) .$
Also, we denote that%
\begin{equation*}
\mathcal{S}_{p}^{\ast }\left( 0\right) =\mathcal{S}_{p}^{\ast },\qquad
\mathcal{S}_{1}^{\ast }\left( \gamma \right) =\mathcal{S}^{\ast }\left(
\gamma \right) \qquad \text{and}\qquad \mathcal{S}_{1}^{\ast }\left(
0\right) =\mathcal{S}^{\ast }.
\end{equation*}

A function $f\in \mathcal{A}_{p}$ is said to be $p$-valently close-to-convex
of order $\gamma \,\left( 0\leq \gamma <p\right) $ in $\mathbb{U}$ if $g\in
\mathcal{S}_{p}^{\ast }\left( \gamma \right) $\ and satisfies the inequality%
\begin{equation*}
\Re \left( \frac{zf^{\prime }\left( z\right) }{g(z)}\right) >\gamma \qquad
\left( z\in \mathbb{U}\right)
\end{equation*}%
or equivalently%
\begin{equation*}
\frac{zf^{\prime }\left( z\right) }{g(z)}\prec \frac{p+\left( p-2\gamma
\right) z}{1-z}\qquad \left( z\in \mathbb{U}\right) .
\end{equation*}%
The class of all $p$-valent close-to-convex functions of order $\gamma $ in $%
\mathbb{U}$ is denoted by $\mathcal{K}_{p}\left( \gamma \right) .$ Also, we
denote that%
\begin{equation*}
\mathcal{K}_{p}\left( 0\right) =\mathcal{K}_{p},\qquad \mathcal{K}_{1}\left(
\gamma \right) =\mathcal{K}\left( \gamma \right) \qquad \text{and}\qquad
\mathcal{K}_{1}\left( 0\right) =\mathcal{K}.
\end{equation*}

We now introduce the following subclass of analytic functions:

\begin{dfn}
\label{def1} Let the function $f$ be analytic in $\mathbb{U}$ and defined by
$\left( \ref{1.1}\right) .$ We say that $f\in \mathcal{K}_{s}^{\left(
k\right) }\left( \gamma ,p\right) ,$ if there exists a function $g\in
\mathcal{S}_{p}^{\ast }\left( \frac{\left( k-1\right) p}{k}\right) $ $(k\in
\mathbb{N}$ is a fixed integer$)$ such that
\begin{equation}
\Re \left( \frac{{z^{\left( k-1\right) p+1}}f^{\prime }(z)}{{g_{k}(z)}}%
\right) >\gamma \quad \left( z\in \mathbb{U};\;0\leq \gamma <p\right) ,
\label{1.4}
\end{equation}%
where $g_{k}$ is defined by the equality
\begin{equation}
g_{k}(z)=\prod_{\nu =0}^{k-1}\varepsilon ^{-\nu p}g(\varepsilon ^{\nu
}z),\quad \varepsilon =e^{2\pi i/k}.  \label{1.5}
\end{equation}
\end{dfn}

By simple calculations we see that the inequality $\left( \ref{1.4}\right) $
is equivalent to%
\begin{equation}
\left\vert \frac{{z^{\left( k-1\right) p+1}}f^{\prime }(z)}{{g_{k}(z)}}%
-p\right\vert <\left\vert \frac{{z^{\left( k-1\right) p+1}}f^{\prime }(z)}{{%
g_{k}(z)}}+p-2\gamma \right\vert .  \label{1.6}
\end{equation}

\noindent \textbf{Remark }$\mathbf{1}$\textbf{. }$\mathbf{(i)}$ For $p=1$,
we get the class $\mathcal{K}_{s}^{\left( k\right) }\left( \gamma ,1\right) =%
\mathcal{K}_{s}^{\left( k\right) }\left( \gamma \right) \,\left( 0\leq
\gamma <1\right) $ recently studied by \c{S}eker \cite{seker}.

\noindent $\mathbf{(ii)}$ For $p=1$ and $k=2$, we have the class $\mathcal{K}%
_{s}^{\left( 2\right) }\left( \gamma ,1\right) =\mathcal{K}_{s}\left( \gamma
\right) \,\left( 0\leq \gamma <1\right) $ introduced and studied by
Kowalczyk and Le\'{s}-Bomba \cite{kow}.

\noindent $\mathbf{(iii)}$ For $p=1$, $k=2$ and $\gamma =0$, we have the
class $\mathcal{K}_{s}^{\left( 2\right) }\left( 0,1\right) =\mathcal{K}_{s}$
introduced and studied by Gao and Zhou \cite{gao}.

In this work, by using the principle of subordination, we obtain inclusion
theorem and distortion theorems for functions in the function class $%
\mathcal{K}_{s}^{\left( k\right) }\left( \gamma ,p\right) .$ Our results
unify and extend the corresponding results obtained by \c{S}eker \cite{seker}%
, Kowalczyk and Le\'{s}-Bomba \cite{kow}, and Gao and Zhou \cite{gao}.

\section{Preliminary Lemmas}

We assume throughout this paper that $\mathcal{P}$ denote the class of
functions $p$ of the form%
\begin{equation*}
p\left( z\right) =1+p_{1}z+p_{2}z^{2}+\cdots \qquad \left( z\in \mathbb{U}%
\right)
\end{equation*}%
which are analytic in $\mathbb{U}$ and $k\in \mathbb{N}$ is a fixed integer.

In order to prove our main results for the functions class $\mathcal{K}%
_{s}^{\left( k\right) }\left( \gamma ,p\right) ,$ we need the following
lemmas.

\begin{lem}
\label{lemma1}If%
\begin{equation}
g(z)=z^{p}+\sum_{n=1}^{\infty }b_{n+p}z^{n+p}\in \mathcal{S}_{p}^{\ast
}\left( \frac{\left( k-1\right) p}{k}\right) ,  \label{2}
\end{equation}%
then
\begin{equation}
G_{k}(z)=\frac{g_{k}(z)}{z^{\left( k-1\right) p}}=z^{p}+\sum_{n=1}^{\infty
}B_{n+p}z^{n+p}\in \mathcal{S}_{p}^{\ast },  \label{2.1}
\end{equation}%
where $g_{k}$ is given by $\left( \ref{1.5}\right) $.
\end{lem}

\begin{proof}
Suppose that%
\begin{equation}
g(z)=z^{p}+\sum_{n=1}^{\infty }b_{n+p}z^{n+p}\in \mathcal{S}_{p}^{\ast
}\left( \frac{\left( k-1\right) p}{k}\right) .  \label{2.a}
\end{equation}%
By $\left( \ref{1.5}\right) $, we have%
\begin{equation}
G_{k}(z)=\frac{g_{k}(z)}{z^{\left( k-1\right) p}}=\frac{\prod_{\nu
=0}^{k-1}\varepsilon ^{-\nu p}g(\varepsilon ^{\nu }z)}{z^{\left( k-1\right)
p}}.  \label{2.b}
\end{equation}%
Differentiating $\left( \ref{2.b}\right) $ logarithmically, we obtain%
\begin{equation}
\frac{zG_{k}^{\prime }(z)}{G_{k}(z)}=\sum_{\nu =0}^{k-1}\frac{\varepsilon
^{\nu }zg^{\prime }(\varepsilon ^{\nu }z)}{g(\varepsilon ^{\nu }z)}-\left(
k-1\right) p.  \label{2.c}
\end{equation}%
From $\left( \ref{2.c}\right) $ together with $\left( \ref{2.a}\right) $, we
get%
\begin{equation*}
\Re \left( \frac{zG_{k}^{\prime }(z)}{G_{k}(z)}\right) =\sum_{\nu
=0}^{k-1}\Re \left( \frac{\varepsilon ^{\nu }zg^{\prime }(\varepsilon ^{\nu
}z)}{g(\varepsilon ^{\nu }z)}\right) -\left( k-1\right) p>0,
\end{equation*}%
which completes the proof of our theorem.
\end{proof}

\begin{lem}
\label{lemma2}Let the function%
\begin{equation*}
H\left( z\right) =p+h_{1}z+h_{2}z^{2}+\cdots \qquad \left( z\in \mathbb{U}%
\right)
\end{equation*}%
be analytic in the unit disk $\mathbb{U}$. Then, the function $H$ satisfies
the condition%
\begin{equation*}
\left\vert \frac{H\left( z\right) -p}{\left( p-2\gamma \right) +H\left(
z\right) }\right\vert <\beta \qquad \left( z\in \mathbb{U}\right)
\end{equation*}%
for some $\beta \;\left( 0<\beta \leq 1\right) $, if and only if there
exists an analytic function $\varphi $ in the unit disk $\mathbb{U}$, such
that $\left\vert \varphi \left( z\right) \right\vert \leq \beta \;\left(
z\in \mathbb{U}\right) $, and%
\begin{equation*}
H\left( z\right) =\frac{p-\left( p-2\gamma \right) z\varphi \left( z\right)
}{1+z\varphi \left( z\right) }\qquad \left( z\in \mathbb{U}\right) .
\end{equation*}
\end{lem}

\begin{proof}
We will employ the technique similar with those of Padamanabhan \cite{P}.
Assume that the function%
\begin{equation*}
H\left( z\right) =p+h_{1}z+h_{2}z^{2}+\cdots \qquad \left( z\in \mathbb{U}%
\right)
\end{equation*}%
satisfies the condition%
\begin{equation*}
\left\vert \frac{H\left( z\right) -p}{\left( p-2\gamma \right) +H\left(
z\right) }\right\vert <\beta \qquad \left( z\in \mathbb{U}\right) .
\end{equation*}%
Setting%
\begin{equation*}
h\left( z\right) =\frac{p-H\left( z\right) }{\left( p-2\gamma \right)
+H\left( z\right) },
\end{equation*}%
we see that the function $h$ analytic in $\mathbb{U}$, satisfies the
inequality $\left\vert h\left( z\right) \right\vert <\beta $ for $z\in
\mathbb{U}$\ and $h(0)=0$. Now, by using the Schwarz's lemma, we get that
the function $h$ has the form $h\left( z\right) =z\varphi \left( z\right) $,
where $\varphi $ is analytic in $\mathbb{U}$ and satisfies $\left\vert
\varphi \left( z\right) \right\vert \leq \beta $ for $z\in \mathbb{U}$.
Thus, we obtain%
\begin{equation*}
H(z)=\frac{p-\left( p-2\gamma \right) h(z)}{1+h(z)}=\frac{p-\left( p-2\gamma
\right) z\varphi \left( z\right) }{1+z\varphi \left( z\right) }.
\end{equation*}%
Conversely, if%
\begin{equation*}
H\left( z\right) =\frac{p-\left( p-2\gamma \right) z\varphi \left( z\right)
}{1+z\varphi \left( z\right) }
\end{equation*}%
and $\left\vert \varphi \left( z\right) \right\vert \leq \beta $ for $z\in
\mathbb{U}$, then $H$ is analytic in the unit disk $\mathbb{U}$. So we get%
\begin{equation*}
\left\vert \frac{H\left( z\right) -p}{\left( p-2\gamma \right) +H\left(
z\right) }\right\vert =\left\vert z\varphi \left( z\right) \right\vert \leq
\beta \left\vert z\right\vert <\beta \qquad \left( z\in \mathbb{U}\right)
\end{equation*}%
which completes the proof of our lemma.
\end{proof}

\begin{lem}
\label{lemma4}\cite{HOS} A function $p\in \mathcal{P}$ satisfies the
following condition:%
\begin{equation*}
\Re \left( p(z)\right) >0\qquad \left( z\in \mathbb{U}\right)
\end{equation*}%
if and only if%
\begin{equation*}
p(z)\neq \frac{\zeta -1}{\zeta +1}\qquad \left( z\in \mathbb{U};\;\zeta \in
\mathbb{C};\;\left\vert \zeta \right\vert =1\right) .
\end{equation*}
\end{lem}

\begin{lem}
\label{lemma5}A function $f\in \mathcal{A}_{p}$ given by $\left( \ref{1.1}%
\right) $ is in the class $\mathcal{K}_{s}^{\left( k\right) }\left( \gamma
,p\right) $ if and only if%
\begin{equation*}
1+\sum_{n=1}^{\infty }A_{n+p}z^{n}\neq 0\qquad \left( z\in \mathbb{U}\right)
,
\end{equation*}%
where%
\begin{equation*}
A_{n+p}=\frac{\left( \zeta +1\right) \left( n+p\right) a_{n+p}+\left(
p-2\gamma -p\zeta \right) B_{n+p}}{2\left( p-\gamma \right) }\qquad \left(
\zeta \in \mathbb{C};\;\left\vert \zeta \right\vert =1\right) .
\end{equation*}
\end{lem}

\begin{proof}
Upon setting%
\begin{equation*}
p(z)=\frac{\frac{{z^{\left( k-1\right) p+1}}f^{\prime }(z)}{{g_{k}(z)}}%
-\gamma }{p-\gamma }\qquad \left( f\in \mathcal{K}_{s}^{\left( k\right)
}\left( \gamma ,p\right) \right) ,
\end{equation*}%
we find that%
\begin{equation*}
p(z)\in \mathcal{P}\qquad \text{and\qquad }\Re \left( p(z)\right) >0\qquad
\left( z\in \mathbb{U}\right) .
\end{equation*}%
Using Lemma $\ref{lemma4}$, we have%
\begin{equation}
\frac{\frac{{z^{\left( k-1\right) p+1}}f^{\prime }(z)}{{g_{k}(z)}}-\gamma }{%
p-\gamma }\neq \frac{\zeta -1}{\zeta +1}\qquad \left( z\in \mathbb{U}%
;\;\zeta \in \mathbb{C};\;\left\vert \zeta \right\vert =1\right) .
\label{2.d}
\end{equation}%
For $z=0$, the above relation holds, since%
\begin{equation*}
\left. \frac{\frac{{z^{\left( k-1\right) p+1}}f^{\prime }(z)}{{g_{k}(z)}}%
-\gamma }{p-\gamma }\right\vert _{z=0}=1\neq \frac{\zeta -1}{\zeta +1}\qquad
\left( \zeta \in \mathbb{C};\;\left\vert \zeta \right\vert =1\right) .
\end{equation*}%
For $z\neq 0$, the relation $\left( \ref{2.d}\right) $ is equivalent to%
\begin{equation*}
\left( \zeta +1\right) {z^{\left( k-1\right) p+1}}f^{\prime }(z)+\left(
p-2\gamma -p\zeta \right) {g_{k}(z)\neq 0\quad }\left( f\in \mathcal{K}%
_{s}^{\left( k\right) }\left( \gamma ,p\right) ;\;\zeta \in \mathbb{C}%
;\;\left\vert \zeta \right\vert =1\right) .
\end{equation*}%
Thus from $\left( \ref{1.1}\right) $ and $\left( \ref{2.1}\right) $\ we find
that%
\begin{equation*}
\left( \zeta +1\right) \left( pz^{p}+\sum_{n=1}^{\infty }\left( n+p\right)
a_{n+p}z^{n+p}\right) +\left( p-2\gamma -p\zeta \right) \left(
z^{p}+\sum_{n=1}^{\infty }B_{n+p}z^{n+p}\right) {\neq 0,}
\end{equation*}%
that is, that%
\begin{equation*}
2\left( p-\gamma \right) z^{p}+\sum_{n=1}^{\infty }\left[ \left( \zeta
+1\right) \left( n+p\right) a_{n+p}+\left( p-2\gamma -p\zeta \right) B_{n+p}%
\right] z^{n+p}{\neq 0.}
\end{equation*}%
Now, dividing both sides of above relation by $2\left( p-\gamma \right)
z^{p}\;\left( z\neq 0\right) $, we obtain%
\begin{equation*}
1+\sum_{n=1}^{\infty }\frac{\left( \zeta +1\right) \left( n+p\right)
a_{n+p}+\left( p-2\gamma -p\zeta \right) B_{n+p}}{2\left( p-\gamma \right) }%
z^{n}{\neq 0}
\end{equation*}%
which completes the proof of Lemma $\ref{lemma5}$.
\end{proof}

\begin{lem}
\label{lemma3}\cite{L} Let $-1\leq B_{2}\leq B_{1}<A_{1}\leq A_{2}\leq 1$.
Then%
\begin{equation*}
\frac{1+A_{1}z}{1+B_{1}z}\prec \frac{1+A_{2}z}{1+B_{2}z}.
\end{equation*}
\end{lem}

\section{Main Results}

We now state and prove the main results of our present investigation.

\begin{thm}
\label{thm1} Let $f$ be an analytic function in $\mathbb{U}$ given by $%
\left( \ref{1.1}\right) $. Then $f\in \mathcal{K}_{s}^{\left( k\right)
}\left( \gamma ,p\right) $ if and only if there exists a function $g\in
\mathcal{S}_{p}^{\ast }\left( \frac{\left( k-1\right) p}{k}\right) $ such
that
\begin{equation}
\frac{{z^{\left( k-1\right) p+1}}f^{\prime }(z)}{{g_{k}(z)}}\prec \frac{%
p+\left( p-2\gamma \right) z}{1-z}\qquad \left( z\in \mathbb{U}\right) ,
\label{2.2}
\end{equation}%
where $g_{k}$ is given by $\left( \ref{1.5}\right) $.
\end{thm}

\begin{proof}
Let $f\in \mathcal{K}_{s}^{\left( k\right) }\left( \gamma ,p\right) $. Then,
there exists a function $g\in \mathcal{S}_{p}^{\ast }\left( \frac{\left(
k-1\right) p}{k}\right) $ such that%
\begin{equation*}
\Re \left( \frac{{z^{\left( k-1\right) p+1}}f^{\prime }(z)}{{g_{k}(z)}}%
\right) >\gamma \quad \left( z\in \mathbb{U};\;0\leq \gamma <p\right)
\end{equation*}%
or equivalently%
\begin{equation*}
\frac{{z^{\left( k-1\right) p+1}}f^{\prime }(z)}{{g_{k}(z)}}\prec \frac{%
p+\left( p-2\gamma \right) z}{1-z}.
\end{equation*}%
Conversely, we assume that the subordination $\left( \ref{2.2}\right) $
holds. Then, there exists an analytic function $w$ in $\mathbb{U}$ such that
$w(0)=0,$ $\left\vert w\left( z\right) \right\vert <1$ and%
\begin{equation*}
\frac{{z^{\left( k-1\right) p+1}}f^{\prime }(z)}{{g_{k}(z)}}=\frac{p+\left(
p-2\gamma \right) w\left( z\right) }{1-w\left( z\right) }.
\end{equation*}%
Hence, by using condition $\left\vert w\left( z\right) \right\vert <1$ we
get $\left( \ref{1.6}\right) $, which is equivalent to $\left( \ref{1.4}%
\right) $, so $f\in \mathcal{K}_{s}^{\left( k\right) }\left( \gamma
,p\right) .$
\end{proof}

\noindent \textbf{Remark }$\mathbf{2}$\textbf{.} Letting $p=1$ in Theorem $%
\ref{thm1}$, we have \cite[Theorem 1]{seker}.

\begin{thm}
\label{thm2}We have%
\begin{equation*}
\mathcal{K}_{s}^{\left( k\right) }\left( \gamma ,p\right) \subset \mathcal{K}%
_{p}.
\end{equation*}
\end{thm}

\begin{proof}
Let $f\in \mathcal{K}_{s}^{\left( k\right) }\left( \gamma ,p\right) $ be an
arbitrary function. From Definition $\ref{def1}$, we obtain%
\begin{equation}
\Re \left( \frac{zf^{\prime }(z)}{\frac{{g_{k}(z)}}{{z^{\left( k-1\right) p}}%
}}\right) >\gamma .  \label{2.3}
\end{equation}%
Note that the condition $\left( \ref{2.3}\right) $ can be written as
\begin{equation*}
\Re \left( \frac{{zf}^{\prime }(z)}{{G_{k}(z)}}\right) >\gamma ,
\end{equation*}%
where ${G_{k}}$ is given by $\left( \ref{2.1}\right) $. By Lemma $\ref%
{lemma1}$ since $G_{k}\in \mathcal{S}_{p}^{\ast }$, from the above
inequality, we deduce that $f\in \mathcal{K}_{p}.$
\end{proof}

\begin{thm}
\label{thm3}Suppose that $g\in \mathcal{S}_{p}^{\ast }\left( \frac{\left(
k-1\right) p}{k}\right) $, where $g_{k}$ is given by $\left( \ref{1.5}%
\right) $. If $f$ is an analytic function in $\mathbb{U}$ of the form $%
\left( \ref{1.1}\right) $, such that%
\begin{equation}
2\sum_{n=1}^{\infty }\left( n+p\right) \left\vert a_{n+p}\right\vert +\left(
\left\vert p-2\gamma \right\vert +p\right) \sum_{n=1}^{\infty }\left\vert
B_{n+p}\right\vert <2\left( p-\gamma \right) ,  \label{2.9}
\end{equation}%
where the coefficients $B_{n+p}$ are given by $\left( \ref{2.1}\right) $,
then $f\in \mathcal{K}_{s}^{(k)}\left( \gamma ,p\right) .$
\end{thm}

\begin{proof}
For $f$ given by $\left( \ref{1.1}\right) $ and $g_{k}$ defined by $\left( %
\ref{1.5}\right) $, we set%
\begin{eqnarray*}
\Lambda &=&\left\vert zf^{\prime }(z)-p\frac{{g_{k}(z)}}{{z^{\left(
k-1\right) p}}}\right\vert -\left\vert zf^{\prime }(z)+\frac{\left(
p-2\gamma \right) {g_{k}(z)}}{{z^{\left( k-1\right) p}}}\right\vert \\
&=&\left\vert \sum_{n=1}^{\infty }\left( n+p\right)
a_{n+p}z^{n+p}-p\sum_{n=1}^{\infty }B_{n+p}z^{n+p}\right\vert \\
&&-\left\vert 2\left( p-\gamma \right) z^{p}+\sum_{n=1}^{\infty }\left(
n+p\right) a_{n+p}z^{n+p}+\left( p-2\gamma \right) \sum_{n=1}^{\infty
}B_{n+p}z^{n+p}\right\vert .
\end{eqnarray*}%
Thus, for $\left\vert z\right\vert =r\,\left( 0\leq r<1\right) $, we get%
\begin{eqnarray*}
\Lambda &\leq &\sum_{n=1}^{\infty }\left( n+p\right) \left\vert
a_{n+p}\right\vert \left\vert z\right\vert ^{n+p}+p\sum_{n=1}^{\infty
}\left\vert B_{n+p}\right\vert \left\vert z\right\vert ^{n+p} \\
&&-\left( 2\left( p-\gamma \right) \left\vert z\right\vert
^{p}-\sum_{n=1}^{\infty }\left( n+p\right) \left\vert a_{n+p}\right\vert
\left\vert z\right\vert ^{n+p}-\left\vert p-2\gamma \right\vert
\sum_{n=1}^{\infty }\left\vert B_{n+p}\right\vert \left\vert z\right\vert
^{n+p}\right) \\
&=&-2\left( p-\gamma \right) \left\vert z\right\vert
^{p}+2\sum_{n=1}^{\infty }\left( n+p\right) \left\vert a_{n+p}\right\vert
\left\vert z\right\vert ^{n+p}+\left( \left\vert p-2\gamma \right\vert
+p\right) \sum_{n=1}^{\infty }\left\vert B_{n+p}\right\vert \left\vert
z\right\vert ^{n+p} \\
&=&-\left( 2\left( p-\gamma \right) +2\sum_{n=1}^{\infty }\left( n+p\right)
\left\vert a_{n+p}\right\vert +\left( \left\vert p-2\gamma \right\vert
+p\right) \sum_{n=1}^{\infty }\left\vert B_{n+p}\right\vert \right)
\left\vert z\right\vert ^{p}.
\end{eqnarray*}%
From inequality $\left( \ref{2.9}\right) $, we obtain that $\Lambda <0$.
Thus we have%
\begin{equation*}
\left\vert zf^{\prime }(z)-p\frac{{g_{k}(z)}}{{z^{\left( k-1\right) p}}}%
\right\vert <\left\vert zf^{\prime }(z)+\frac{\left( p-2\gamma \right) {%
g_{k}(z)}}{{z^{\left( k-1\right) p}}}\right\vert ,
\end{equation*}%
which is equivalent to $\left( \ref{1.6}\right) $. Hence $f\in \mathcal{K}%
_{s}^{(k)}\left( \gamma ,p\right) $. This completes the proof of Theorem $%
\ref{thm3}$.
\end{proof}

\noindent \textbf{Remark }$\mathbf{3}$\textbf{.} Letting $p=1$ in Theorem $%
\ref{thm3}$, we have \cite[Theorem 2]{seker}.

\begin{thm}
\label{thm4}Suppose that an analytic function $f$ given by $\left( \ref{1.1}%
\right) $ and $g\in \mathcal{S}_{p}^{\ast }\left( \frac{\left( k-1\right) p}{%
k}\right) $ given by $\left( \ref{2}\right) $ are such that the condition $%
\left( \ref{1.4}\right) $ holds. Then, for $n\geq 1$, we have%
\begin{equation}
\left\vert \left( n+p\right) a_{n+p}-pB_{n+p}\right\vert ^{2}-4\left(
p-\gamma \right) ^{2}\leq 2\left( p-\gamma \right)
\sum_{m=p+1}^{n+p-1}\left\{ 2m\left\vert a_{m}B_{m}\right\vert +\left(
\left\vert p-2\gamma \right\vert +p\right) \left\vert B_{m}\right\vert
^{2}\right\} ,  \label{2.7}
\end{equation}%
where the coefficients $B_{n+p}$ are given by $\left( \ref{2.1}\right) $.
\end{thm}

\begin{proof}
Suppose that the condition $\left( \ref{1.4}\right) $ is satisfied. Then
from Lemma $\ref{lemma2}$, we have%
\begin{equation*}
\frac{{zf}^{\prime }(z)}{{G_{k}(z)}}=\frac{p-\left( p-2\gamma \right)
z\varphi \left( z\right) }{1+z\varphi \left( z\right) }\qquad \left( z\in
\mathbb{U}\right) ,
\end{equation*}%
where $\varphi $ is an analytic functions in $\mathbb{U}$, $\left\vert
\varphi \left( z\right) \right\vert \leq 1$ for $z\in \mathbb{U}$, and ${%
G_{k}}$ is given by $\left( \ref{2.1}\right) $. From the above equality, we
obtain that%
\begin{equation}
\left[ {zf}^{\prime }(z)+\left( p-2\gamma \right) {G_{k}(z)}\right] z\varphi
\left( z\right) =p{G_{k}(z)-zf}^{\prime }(z).  \label{2.5}
\end{equation}%
Now, we put%
\begin{equation*}
z\varphi \left( z\right) :=\sum_{n=1}^{\infty }t_{n}z^{n}\qquad \left( z\in
\mathbb{U}\right) .
\end{equation*}%
Thus from $\left( \ref{2.5}\right) $ we find that%
\begin{eqnarray}
&&\left( 2\left( p-\gamma \right) z^{p}+\sum_{n=1}^{\infty }\left(
n+p\right) a_{n+p}z^{n+p}+\left( p-2\gamma \right) \sum_{n=1}^{\infty
}B_{n+p}z^{n+p}\right) \sum_{n=1}^{\infty }t_{n}z^{n}  \notag \\
&=&p\sum_{n=1}^{\infty }B_{n+p}z^{n+p}-\sum_{n=1}^{\infty }\left( n+p\right)
a_{n+p}z^{n+p}.  \label{2.6}
\end{eqnarray}%
Equating the coefficient of $z^{n+p}$ in $\left( \ref{2.6}\right) $, we have%
\begin{eqnarray*}
pB_{n+p}-\left( n+p\right) a_{n+p} &=&2\left( p-\gamma \right) t_{n}+\left[
\left( p+1\right) a_{p+1}+\left( p-2\gamma \right) B_{p+1}\right] t_{n-1} \\
&&+\cdots +\left[ \left( n+p-1\right) a_{n+p-1}+\left( p-2\gamma \right)
B_{n+p-1}\right] t_{1}.
\end{eqnarray*}%
Note that the coefficient combination on the right side of $\left( \ref{2.6}%
\right) $ depends only upon the coefficients combinations:%
\begin{equation*}
\left[ \left( p+1\right) a_{p+1}+\left( p-2\gamma \right) B_{p+1}\right]
,\ldots ,\left[ \left( n+p-1\right) a_{n+p-1}+\left( p-2\gamma \right)
B_{n+p-1}\right] .
\end{equation*}%
Hence, for $n\geq 1$, we can write as%
\begin{eqnarray*}
\left( 2\left( p-\gamma \right) z^{p}+\sum_{m=p+1}^{n+p-1}\left[
ma_{m}+\left( p-2\gamma \right) B_{m}\right] z^{m}\right) z\varphi \left(
z\right)  &=&\sum_{m=p+1}^{n+p}\left[ pB_{m}-ma_{m}\right] z^{m} \\
&&+\sum_{m=n+p+1}^{\infty }c_{m}z^{m}.
\end{eqnarray*}%
Using the fact that $\left\vert z\varphi \left( z\right) \right\vert \leq
\left\vert z\right\vert <1$ for all $z\in \mathbb{U}$, this reduces to the
inequality%
\begin{equation*}
\left\vert 2\left( p-\gamma \right) z^{p}+\sum_{m=p+1}^{n+p-1}\left[
ma_{m}+\left( p-2\gamma \right) B_{m}\right] z^{m}\right\vert >\left\vert
\sum_{m=p+1}^{n+p}\left[ pB_{m}-ma_{m}\right] z^{m}+\sum_{m=n+p+1}^{\infty
}c_{m}z^{m}\right\vert .
\end{equation*}%
Then squaring the above inequality and integrating along $\left\vert
z\right\vert =r<1$, we obtain%
\begin{eqnarray*}
&&\int_{0}^{2\pi }\left\vert 2\left( p-\gamma \right) r^{p}e^{ip\theta
}+\sum_{m=p+1}^{n+p-1}\left[ ma_{m}+\left( p-2\gamma \right) B_{m}\right]
r^{m}e^{im\theta }\right\vert ^{2}d\theta  \\
&>&\int_{0}^{2\pi }\left\vert \sum_{m=p+1}^{n+p}\left[ pB_{m}-ma_{m}\right]
r^{m}e^{im\theta }+\sum_{m=n+p+1}^{\infty }c_{m}r^{m}e^{im\theta
}\right\vert ^{2}d\theta .
\end{eqnarray*}%
Using now the Parseval's inequality, we obtain%
\begin{equation*}
4\left( p-\gamma \right) ^{2}r^{2p}+\sum_{m=p+1}^{n+p-1}\left\vert
ma_{m}+\left( p-2\gamma \right) B_{m}\right\vert
^{2}r^{2m}>\sum_{m=p+1}^{n+p}\left\vert ma_{m}-pB_{m}\right\vert
^{2}r^{2m}+\sum_{m=n+p+1}^{\infty }\left\vert c_{m}\right\vert ^{2}r^{2m}.
\end{equation*}%
Letting $r\rightarrow 1$ in this inequality, we get%
\begin{equation*}
\sum_{m=p+1}^{n+p}\left\vert ma_{m}-pB_{m}\right\vert ^{2}\leq 4\left(
p-\gamma \right) ^{2}+\sum_{m=p+1}^{n+p-1}\left\vert ma_{m}+\left( p-2\gamma
\right) B_{m}\right\vert ^{2}.
\end{equation*}%
Hence we deduce that%
\begin{equation*}
\left\vert \left( n+p\right) a_{n+p}-pB_{n+p}\right\vert ^{2}-4\left(
p-\gamma \right) ^{2}\leq 2\left( p-\gamma \right)
\sum_{m=p+1}^{n+p-1}\left\{ 2m\left\vert a_{m}B_{m}\right\vert +\left(
\left\vert p-2\gamma \right\vert +p\right) \left\vert B_{m}\right\vert
^{2}\right\} ,
\end{equation*}%
and thus we obtain the inequality $\left( \ref{2.7}\right) $, which
completes the proof of Theorem $\ref{thm4}$.
\end{proof}

\noindent \textbf{Remark }$\mathbf{4}$\textbf{.} Letting $p=1$ in Theorem $%
\ref{thm4}$, we have \cite[Theorem 3]{seker}.

\begin{thm}
\label{thm5}If $f\in \mathcal{K}_{s}^{(k)}\left( \gamma ,p\right) $, then
for $\left\vert z\right\vert =r$\ $(0\leq r<1)$, we have

$(i)$%
\begin{equation}
\frac{\left[ p-\left( p-2\gamma \right) r\right] r^{p-1}}{\left( 1+r\right)
^{2p+1}}\leq \left\vert {f}^{\prime }(z)\right\vert \leq \frac{\left[
p+\left( p-2\gamma \right) r\right] r^{p-1}}{\left( 1-r\right) ^{2p+1}},
\label{2.10}
\end{equation}

$(ii)$%
\begin{equation}
\int_{0}^{r}\frac{\left[ p-\left( p-2\gamma \right) \tau \right] \tau ^{p-1}%
}{\left( 1+\tau \right) ^{2p+1}}d\tau \leq \left\vert f(z)\right\vert \leq
\int_{0}^{r}\frac{\left[ p+\left( p-2\gamma \right) \tau \right] \tau ^{p-1}%
}{\left( 1-\tau \right) ^{2p+1}}d\tau ,  \label{2.11}
\end{equation}
\end{thm}

\begin{proof}
If $f\in \mathcal{K}_{s}^{(k)}\left( \gamma ,p\right) $, then there exists a
function $g\in \mathcal{S}_{p}^{\ast }\left( \frac{\left( k-1\right) p}{k}%
\right) $ such that $\left( \ref{1.4}\right) $ holds.

\noindent $(i)$ From Lemma $\ref{lemma1}$ it follows that the function $%
G_{k} $ given by $\left( \ref{2.1}\right) $ is $p$-valently starlike
function. Hence from \cite[Theorem 1]{A} we have%
\begin{equation}
\frac{r^{p}}{\left( 1+r\right) ^{2p}}\leq \left\vert G_{k}(z)\right\vert
\leq \frac{r^{p}}{\left( 1-r\right) ^{2p}}\qquad \left( \left\vert
z\right\vert =r\;(0\leq r<1)\right) .  \label{2.13}
\end{equation}%
Let us define $\Psi (z)$ by%
\begin{equation*}
\Psi (z):=\frac{{zf}^{\prime }(z)}{{G_{k}(z)}}\qquad \left( z\in \mathbb{U}%
\right) .
\end{equation*}%
Then by using a similar method \cite[p. 105]{G}, we have%
\begin{equation}
\frac{p-\left( p-2\gamma \right) r}{1+r}\leq \left\vert \Psi (z)\right\vert
\leq \frac{p+\left( p-2\gamma \right) r}{1-r}\qquad \left( z\in \mathbb{U}%
\right) .  \label{2.14}
\end{equation}%
Thus from $\left( \ref{2.13}\right) $ and $\left( \ref{2.14}\right) $, we
get the inequalities $\left( \ref{2.10}\right) $.

\noindent $(ii)$ Let $z=re^{i\theta }\;(0<r<1)$. If $\ell $ denotes the
closed line-segment in the complex $\zeta $-plane from $\zeta =0$ and $\zeta
=z$, i.e. $\ell =\left[ 0,re^{i\theta }\right] $, then we have%
\begin{equation*}
f(z)=\int_{\ell }f^{\prime }(\zeta )d\zeta =\int_{0}^{r}f^{\prime }(\tau
e^{i\theta })e^{i\theta }d\tau \qquad \left( \left\vert z\right\vert
=r\;(0\leq r<1)\right) .
\end{equation*}%
Thus, by using the upper estimate in $\left( \ref{2.10}\right) $, we have%
\begin{equation*}
\left\vert f(z)\right\vert =\left\vert \int_{\ell }f^{\prime }(\zeta )d\zeta
\right\vert \leq \int_{0}^{r}\left\vert f^{\prime }(\tau e^{i\theta
})\right\vert d\tau \leq \int_{0}^{r}\frac{\left[ p+\left( p-2\gamma \right)
\tau \right] \tau ^{p-1}}{\left( 1-\tau \right) ^{2p+1}}d\tau \qquad \left(
\left\vert z\right\vert =r\;(0\leq r<1)\right) ,
\end{equation*}%
which yields the right-hand side of the inequality in $\left( \ref{2.11}%
\right) $. In order to prove the lower bound in $\left( \ref{2.11}\right) $,
let $z_{0}\in \mathbb{U}$ with $\left\vert z_{0}\right\vert =r\;(0<r<1)$,
such that%
\begin{equation*}
\left\vert f(z_{0})\right\vert =\min \left\{ \left\vert f(z)\right\vert
:\left\vert z\right\vert =r\right\} .
\end{equation*}%
It is sufficient to prove that the left-hand side inequality holds for this
point $z_{0}$. Moreover, we have%
\begin{equation*}
\left\vert f(z)\right\vert \geq \left\vert f(z_{0})\right\vert \qquad \left(
\left\vert z\right\vert =r\;(0\leq r<1)\right) .
\end{equation*}%
The image of the closed line-segment $\ell _{0}=\left[ 0,f(z_{0})\right] $
by $f^{-1}$ is a piece of arc $\Gamma $ included in the closed disk $\mathbb{%
U}_{r}$ given by%
\begin{equation*}
\mathbb{U}_{r}=\left\{ z:z\in \mathbb{C}\quad \text{and}\quad \left\vert
z\right\vert \leq r\;(0\leq r<1)\right\} ,
\end{equation*}%
that is, $\Gamma =f^{-1}\left( \ell _{0}\right) \subset \mathbb{U}_{r}$.
Hence, in accordance with $\left( \ref{2.10}\right) $, we obtain%
\begin{equation*}
\left\vert f(z_{0})\right\vert =\int_{\ell _{0}}\left\vert dw\right\vert
=\int_{\Gamma }\left\vert {f}^{\prime }(\zeta )\right\vert \left\vert d\zeta
\right\vert \geq \int_{0}^{r}\frac{\left[ p-\left( p-2\gamma \right) \tau %
\right] \tau ^{p-1}}{\left( 1+\tau \right) ^{2p+1}}d\tau .
\end{equation*}%
This finishes the proof of the inequality $\left( \ref{2.11}\right) $.
\end{proof}

\noindent \textbf{Remark }$\mathbf{5}$\textbf{.} Letting $p=1$ in Theorem $%
\ref{thm5}$, we have \cite[Theorem 4]{seker}.

\begin{thm}
\label{thm6}Let $0\leq \gamma _{2}\leq \gamma _{1}<p.$ Then we have%
\begin{equation*}
\mathcal{K}_{s}^{(k)}\left( \gamma _{1},p\right) \subset \mathcal{K}%
_{s}^{(k)}\left( \gamma _{2},p\right) .
\end{equation*}
\end{thm}

\begin{proof}
Suppose that $f\in \mathcal{K}_{s}^{(k)}\left( \gamma _{1},p\right) $. By
Theorem $\ref{thm1}$, we have%
\begin{equation*}
\frac{{z^{\left( k-1\right) p+1}}f^{\prime }(z)}{{g_{k}(z)}}\prec \frac{%
p+\left( p-2\gamma _{1}\right) z}{1-z}\qquad \left( z\in \mathbb{U}\right) .
\end{equation*}%
Since $0\leq \gamma _{2}\leq \gamma _{1}<p$, we\ get%
\begin{equation*}
-1<1-\frac{2\gamma _{1}}{p}\leq 1-\frac{2\gamma _{2}}{p}\leq 1.
\end{equation*}%
Thus by Lemma $\ref{lemma3}$, we have%
\begin{equation*}
\frac{{z^{\left( k-1\right) p+1}}f^{\prime }(z)}{{g_{k}(z)}}\prec \frac{%
p+\left( p-2\gamma _{2}\right) z}{1-z}\qquad \left( z\in \mathbb{U}\right) ,
\end{equation*}%
that is $f\in \mathcal{K}_{s}^{(k)}\left( \gamma _{2},p\right) $. Hence the
proof is complete.
\end{proof}

\noindent \textbf{Remark }$\mathbf{6}$\textbf{.} Letting $p=1$ in Theorem $%
\ref{thm6}$, we have \cite[Theorem 5]{seker}.

\end{document}